\documentclass[12pt, a4paper]{amsart}
\usepackage{headers}
\title{Multiplicity one theorem over characteristic 2}
\author{Dor Mezer}
\date{\today}
\begin{document}
\begin{abstract}
It is shown for all local fields $\FF$ which are of characteristic different from $2$ that any distribution on $GL_{n+1}(\FF)$ which is invariant under conjugation by $GL_n(\FF)$ is also invariant under transposition.
In this paper we give an adaptation of the proof of this theorem to fields of characteristic 2.
\end{abstract}
\maketitle
\section{Introduction}
Let $\FF$ be a local field of characteristic $2$.
In this paper we prove the following theorem:
\begin{theorem}\label{goal}
Any distribution on $GL_{n+1}(\FF)$ invariant under conjugation by $GL_n(\FF)$ is also invariant under transposition.
\end{theorem}
For non-archimedean fields of characteristic zero it is proven in \cite{AGRS}, for archimedean fields in \cite{SZ} and \cite{AG-R},and for fields of odd characteristic in \cite{M}. In this paper we will give an adaptation of the proof in \cite{M} to characteristic 2.

It is shown in \cite[section 1]{AGRS} that Theorem \ref{goal} has the following corollary, already known by different methods (see \cite{AGS}).
\begin{theorem}
Let $\pi$ be an irreducible smooth representation of $GL_{n+1}$, and let $\rho$ be an irreducible smooth representation of $GL_n$. Then
$$\dim \operatorname{Hom}_{GL_n}(\pi, \rho)\leq 1$$
\end{theorem}
Let $V$ be an $n$-dimensional vector space over $\FF$.
Let $\tG := GL(V)\rtimes \{\pm 1\}$ be the semidirect product with the respect to the action of $\{\pm 1\}$ on $GL(V)$ by $A\mapsto \left(A^t\right)^{-1}$. The group $\tG$ acts on $\mathfrak{gl}(V)$ by $(g,1).(A,v,\phi) = (gAg^{-1},gv,(g^*)^{-1}\phi)$, and $(g,-1).(A,v,\phi) = (\left(gAg^{-1}\right)^t, g\phi^t, (g^*)^{-1}v^t)$. Let $\chi$ be the character of $\tG$ defined by $(g,\delta)\mapsto \delta$.
It is shown in \cite{AG} (the same proof works verbatim) that Theorem \ref{goal} reduces to the following theorem:

\begin{theorem}\label{main}
Any $(\tG,\chi)$-equivariant distribution on $\mathfrak{gl}(V)\times V\times V^*$ is 0.
\end{theorem}

We will prove this theorem by induction on the dimension of $V$, and so we will assume this theorem for all smaller $n$. Throughout the paper, let $\xi$ be a $(\tG, \chi)$-equivariant distribution on $\mathfrak{gl}(V)\times V\times V^*$.

There are two points in the proof in \cite{M} in which the assumption $\mathrm{char}(\FF)\neq 2$ was made use of. The first and more significant one is the proof of Proposition 4.6. The main goal of this paper is to prove this theorem over fields of characteristic 2 (Corollary \ref{all k}). In section 2 we use the techinique of Harish-chandra descent to a further extent than was used in \cite{M}, to get a stronger restriction on the support of $\xi$, which will be used in the proof of Corollary \ref{all k} in section 3.
The second usage of the assumption $\mathrm{char}(\FF)\neq 2$ in \cite{M} was the usage of a theorem by Rallis and Schiffman (\cite[Theorem 2.9]{M}), which is relied on the theory of the Weil representation, a theory which is a bit different when working over a field of characteristic 2. In Appendix \ref{weilrep}, we prove a version of the theorem over a field of characteristic 2, which is sufficient for the proof as given in \cite{M}.
\subsection{Acknowledgements}
I would like to deeply thank my advisor, Dmitry Gourevitch, for helping me with my work on this paper, for exposing me to this fascinating area of mathematics, and for the guidance and support along the way. I would also like to thank him for the exceptional willingness to help.\\
I would also like to thank Guy Henniart for fruitful discussions he had with my advisor laying the foundations for this project, and for his continued interest in my research along the way.\\
I am also grateful to Shamgar Gurevich, for the help he gave me when writing this paper.\\
D.M. was partially supported by ERC StG grant 637912. 
\section{Reduction to the purely inseparable locus}
\begin{notation}
Use $\Delta:\mathfrak{gl}(V)\times A\times A^*\to \FF[x]$ to denote the map which sends $(A,v,\phi)$ to the characteristic polynomial of $A$.
\end{notation}
\begin{theorem}
For any point $(A,v,\phi)$ in the support of $\xi$, the characteristic polynomial of $A$ is a power of an irreducible polynomial.
\end{theorem}
\begin{proof}
Assume that a polynomial $f$ of degree $n$ has two coprime components $f = f_1f_2$. By the localization principle (see \cite[Theorem 2.4]{M}), it is enough for us to show that for any such polynomial $f$, the fiber of $\Delta$ above $f$ has no non-zero $(\tG,\chi)$-equivariant distributions. Let $\zeta$ be such a distribution on $\Delta^{-1}(f)$. Let $d_1 = \deg f_1, d_2 = \deg f_2$. Denote by $\Lambda$ the space of all pairs of subspaces $V_1,V_2$ of $V$ such that $\dim(V_1) = d_1, \dim(V_2) = d_2, V = V_1\oplus V_2$. $\Lambda$ has a natural action of $G$ on it, which extends to an action of $\tG$ by the involution $(V_1,V_2)\mapsto (V_2^\perp, V_1^\perp)$with respect to the quadratic form $v\mapsto v^t v$.
These two actions are (both) transitive.
There is a $\tG$-equivariant map $\Delta^{-1}(f)\to \Lambda$, given by taking the (unique) pair of $A$-invariant subspaces of $V$ on which $A$ acts with characteristic polynomials $f_1$ and $f_2$. The fiber of this map above $(V_1,V_2)$ is a closed subspace of $(\mathfrak{gl}(V_1)\times V_1 \times V_1^*)\times (\mathfrak{gl}(V_2)\times V_2 \times V_2^*)$, and the stabilizer of this point is equal to $\tG(V_1)\times \tG(V_2)$. By the localization principle (see \cite[Theorem 2.4]{M}) and induction hypothesis, there are no non-zero $(\tG(V_1)\times \tG(V_2), \chi)$-equivariant distributions on $(\mathfrak{gl}(V_1)\times V_1 \times V_1^*)\times (\mathfrak{gl}(V_2)\times V_2 \times V_2^*)$, and so it follows from Frobenius descent (see \cite[Theorem 2.7]{M}) that $\zeta = 0$ too.
\end{proof}
\begin{theorem}\label{inseperable}
For any point $(A,v,\phi)$ in the support of $\xi$, the irreducible factor in the characteristic polynomial of $A$ is purely inseparable.
\end{theorem}
\begin{proof}
By the localization principle (see \cite[Theorem 2.4]{M}) it is enough to consider a $(\tG,\chi)$-equivariant distribution $\zeta$ on $F\times V\times V^*$, where $F$ is the fiber of $\Delta$ over some $f^m$, $f$ being irreducible and not purely separable, and show that $\zeta=0$.
We have $f(x) = g(x^{2^k})$ for $g(x)$ irreducible and separable. By assumption, $\deg g > 1$. For any $A\in F$, the characteristic polynomial of $B:=A^{2^k}$ is equal to $g(x)^{2^k m}$, as it is the polynomial whose roots are $2^k$ powers of the roots of $g(x^{2^k})^m$. All of its irreducible factors are separable, and so $B$ has a well defined Jordan decomposition $B = B_s + B_n$. Moreover, the map $h:A\mapsto B_s$ is continuous on $F$. Note that since $B_s$ is expressable as a polynomial in $B$ (and thus as a polynomial in $A$), it commutes with $A$.
Use Frobenius descent (see \cite[Theorem 2.7]{M}) with respect to $h$ - the stabilizer of a point is isomorphic to $\widetilde{GL}_{2^k m}(\mathbb{E})$, where $\mathbb{E}:= \FF[x] / g(x)$. The fiber of $h$ above a point is isomorphic to a closed subspace of $\mathfrak{gl}_{2^k m}(\mathbb{E})$, and so by induction hypothesis applied to the Frobenius descent of $\zeta$, we get that $\zeta=0$.
\end{proof}

\section{Vanishing of linear invariants}
The following proposition is proved in \cite[Lemma 7.2]{AG} over a field of characteristic $0$, and the proof there applies verbatim over arbitrary characteristic.
\begin{proposition}\label{k=0}
For any point $(A,v,\phi)$ in the support of $\xi$, we have $<v,\phi> = 0$.
\end{proposition}
\begin{definition}
Let $\mu\in\FF$. Define $\rho_\mu$ as the following $\mathrm{GL}(V)$-equivariant automorphism on $\mathfrak{gl}(V)\times V\times V^*$:
$$(A,v,\phi)\mapsto (A+\mu v\otimes \phi, v, \phi).$$
\end{definition}
\begin{theorem}\label{k=1}
For any point $(A,v,\phi)$ in the support of $\xi$, we have $<Av,\phi> = 0$.
\end{theorem}
\begin{proof}
By applying Theorem \ref{inseperable} to $\rho_\mu(\xi)$ for some $\mu \in \FF$, we get that the characteristic polynomial of $A+\mu v\otimes \phi$ must be a power of an irreducible purely inseparable polynomial too.
Denote the characteristic polynomial of $A+\mu v\otimes \phi$ by $\sum_{i=0}^n c_i(A+\mu v\otimes \phi) x^{n-i}$.
\begin{case}
$n$ is odd
\end{case}
The characteristic polynomial of $A+\mu v\otimes \phi$ is of the form $(x+\lambda_\mu)^n$. Since $\lambda_\mu = c_1(A+\mu v\otimes \phi) = c_1(A) - \mu <v,\phi> = c_1(A) = \lambda_0$, we get that $\lambda_\mu$ (and so also the characteristic polynomial of $A+\mu v\otimes \phi$) is independent of $\mu$. Thus we get that
$$c_2(A) = c_2(A+\mu v\otimes \phi)= c_2(A) - \mu(<Av, \phi>+c_1(A)<v,\phi>) = c_2(A) - \mu<Av,\phi>,$$
and so $<Av,\phi> = 0$. 
\begin{case}
$n$ is divisible by $4$
\end{case}
In this case $n=\binom{n}{2} = 0$ in $\mathbb{F}$. The characteristic polynomial of $A+\mu v\otimes \phi$ is always of the form $(x^{2^k}+\lambda)^{\frac{n}{2^k}}$ for some $\lambda$ dependent on $\mu$. By maybe changing $\lambda$, we can assume that $2^k$ is the maximal power of $2$ that divides $n$. In particular our polynomial is a polynomial in $x^4$, and so we have $c_2(A+\mu v\otimes \phi) = 0$ for all $\mu$.
However,
$$c_2(A+\mu v\otimes \phi) = c_2(A) - \mu(<Av,\phi> + c_1(A)<v,\phi>) = c_2(A) - \mu <Av,\phi>.$$ Thus we must have $<Av,\phi> = 0$.
\begin{case}
$n = 2 \mod 4$ and $n>2$
\end{case}
The irreducible factor of the characteristic polynomial is either linear or quadratic. Thus the characteristic polynomial must be either $(x+\lambda)^n$ or $(x^2 + \lambda)^{n/2}$. Allowing $\lambda$ to be a square, we assume it is of the second form. So $c_2$ is equal $(n/2)\lambda = \lambda$. Let $\lambda_\mu$ be such that the characteristic polynomial of $A+\mu v\otimes \phi$ is $(x^2+\lambda_\mu)^{n/2}$. We have
$$\lambda_\mu = c_2(A+\mu v\otimes \phi) = c_2(A) - \mu <Av,\phi>$$
$$(c_2(A) - \mu(<Av,\phi>))^{n/2} = \lambda_\mu^{n/2} = c_n(A+\mu v\otimes \phi) = c_n(A) - \mu(\dots)$$
The right hand side is linear in $\mu$ while the left hand side is polynomial of degree $n/2$ unless $<Av,\phi>=0$. Thus assuming $n>2$ we are done.
\begin{case}
$n=2$
\end{case}
We can use the localization principle (see \cite[Theorem 2.4]{M}) with respect to the map $(A,v,\phi)\mapsto <Av,\phi>$, to be left with proving that if $\xi$ is supported on $\{(A,v,\phi)|<Av,\phi>=m\}$ (for some $m\neq 0$) then $\xi=0$.
We already know that for any point $(A,v,\phi)$ in the support of $\xi$, we must have $\mathrm{tr}A=0$, and that $<v,\phi>=0$. Recalling that by the assumption $<Av,\phi>=m$ we necessarily have $v\neq 0$, we can write explicitly $A = \begin{pmatrix} a & b\\ c & a\end{pmatrix}, v = \begin{pmatrix} x\\ y\end{pmatrix}, \phi = \begin{pmatrix}ty & tx\end{pmatrix}$. This yields $<Av,\phi>=t(cx^2+by^2)$. Let $\sigma:=\left(\begin{pmatrix} 0&1\\ 1&0 \end{pmatrix}, -1\right)\in \tG$.
It acts by
$$\left( \begin{pmatrix} a & b\\ c & a\end{pmatrix}, \begin{pmatrix} x\\ y\end{pmatrix}, \begin{pmatrix}ty & tx\end{pmatrix}\right) \mapsto \left( \begin{pmatrix} a & b\\ c & a\end{pmatrix}, \begin{pmatrix} tx\\ ty\end{pmatrix}, \begin{pmatrix}y & x\end{pmatrix}\right).$$
Use Frobenius descent (see \cite[Theorem 2.7]{M}) with respect to the $\tG$-equivariant map
$$\left\{\left( \begin{pmatrix} a & b\\ c & a\end{pmatrix}, \begin{pmatrix} x\\ y\end{pmatrix}, \begin{pmatrix}ty & tx\end{pmatrix}\right)|\begin{pmatrix}ty & tx\end{pmatrix}\neq 0\right\}\to \left\{\left( \begin{pmatrix} x\\ y\end{pmatrix}, \begin{pmatrix}ty & tx\end{pmatrix}\right)|\begin{pmatrix}ty & tx\end{pmatrix}\neq 0\right\}$$
given in the above coordinates.
It is easy to see that the action on the target is transitive. The stabilizer of the point $\left( \begin{pmatrix} 1\\ 0\end{pmatrix}, \begin{pmatrix}0 & 1\end{pmatrix}\right)$ inside $G$ is the unimodular subgroup $N:=\left\{\begin{pmatrix} 1&*\\0&1\end{pmatrix}\right\}$. The stabilizer inside $\tG$ is $\tilde{N}:= N\rtimes\{1,\sigma\}$, and it is indeed also unimodular.
The fiber above $\left( \begin{pmatrix} 1\\ 0\end{pmatrix}, \begin{pmatrix}0 & 1\end{pmatrix}\right)$ is $\left\{\begin{pmatrix}a&b\\m&a\end{pmatrix}\right\}$, on which $\sigma$ acts trivially. Thus, since our distribution is $\chi$-equivariant (and so $\sigma$-anti-invariant), it must be 0. 

\end{proof}

The following Corollary of the previous theorem appears in \cite[Proposition 4.6]{M} and is proved there in a way which fails over a field of characteristic $2$.
\begin{corollary}\label{all k}
For any point $(A,v,\phi)$ in the support of $\xi$ and any $k\geq 0$, we have $<A^k v,\phi> = 0$.
\end{corollary}
\begin{proof}
Let $\Delta: \mathfrak{gl}(V)\times V \times V^*\to \FF[x]$ be the characteristic polynomial map.
Recall the automorphism $\rho_g:\Delta^{-1}(f)\to \Delta^{-1}(f)$ defined for every $g$ coprime to $f$ by $\rho_g((A,v,\phi)):=(A,g(A)v,g(A^*)\phi)$. By using the localization principle (see \cite[Theorem 2.4]{M}), we can reduce the claim to a distribution on a single fiber of $\Delta$ over a polynomial $f$. Then for any $g$ coprime to $f$, we can apply $\rho_g$, then extend the distribution back to $\mathfrak{gl}(V)\times V\times V^*$ and apply Theorem \ref{k=0} to get that $<g(A)v,g(A^*)\phi> = 0$ and Theorem \ref{k=1} to get that $<Ag(A)v,g(A^*)\phi> = 0$. Since this is true for a Zariski dense set of polynomials $g$, it is true for all polynomials.
Thus we get that for any $g$, we have $<g(A)^2v,\phi> = <g(A),g(A^*)\phi> = 0$ and $<Ag(A)^2v,\phi> = <Ag(A),g(A^*)\phi> = 0$. In particular for any $k\geq 0$ we can take $g(x) = x^k$ to get that $<A^{2k}v,\phi> = 0$ and $<A^{2k+1}v,\phi> = 0$.
\end{proof}
Once this Corollary is proven, the rest of the proof of Theorem \ref{main} given in \cite{M} applies almost verbatim also over a field of characteristic $2$. The only point in which there is a difference is the usage of \cite[Theorem 2.9]{M} of which the proof relies on the theory of the Weil representation. This theory is a bit different over characteristic 2. We give in Appendix \ref{weilrep} the necessary adaptations, and prove Theorem \ref{RS thm} which plays the same role as \cite[Theorem 2.9]{M}.

\appendix
\section{The Weil representation over characteristic 2}\label{weilrep}
In this section we prove the following theorem, due to \cite{RS} in the case where $\mathrm{char} \FF\neq 2$:
\begin{theorem}\label{RS thm}
Let $\FF$ be a local field with charateristic $2$. Let $V$ be a finite dimensional linear space over $\FF$, and let $V^*$ be its dual space. Define $Z:=\{(v,\phi)\in V\times V^*|<v,\phi>=0\}$. Then for every distribution $\xi$ on $V\times V^*$ such that both $\xi$ and its Fourier transform $\Fou(\xi)$ are supported on $Z$, we have that $\xi$ must be "abs-homogeneous" of degree $\dim V$. That is, for any $t\in \FF$ and $\Phi\in \Sc(V\times V^*)$, we have $|(t\xi)(\Phi)| = |t|^{\dim V}\cdot |\xi(\Phi)|$.
\end{theorem}
Let us recall the facts which we will need about the Weil representation in characteristic $2$ and the pseudo-symplectyc group (see \cite{Weil} and \cite{Bla}).
\begin{definition}
Let $\FF,V$ as above. We define $X:=V\oplus V^*$. Define a bilinear form $B$ on $X\oplus X^*\cong V\oplus V^*\oplus V\oplus V^*$ by
$$B((u_1, u'_1, v_1, v'_1), (u_2, u'_2, v_2, v'_2))= <u_1, v'_2>+<v_2,u'_1>.$$
Let $Q$ be the quadratic form on $X\oplus X^*$ defined by $B$, and denote by $\mathfrak{Q}(X\oplus X^*)$ the linear space of all quadratic forms on $X\oplus X^*$.\\
Define the pseudo-symplectic group
\begin{align*}
PSp:=\{&(\sigma, f)\in O(Q)\times \mathfrak{Q}(X\oplus X^*)|\forall w,w'\in X\oplus X^*,\\
&B(w,w')+B(\sigma w,\sigma w')=f(w+w')+f(w)+f(w')\}
\end{align*}
with group law $(\sigma_1,f_1)(\sigma_2,f_2)=(\sigma_1\sigma_2,f_1\sigma_2+f_2)$.
\end{definition}
\begin{proposition}
We have an embedding $j:SL_2(\FF)\to PSp$ defined by
$$\begin{pmatrix} a&b\\c&d \end{pmatrix}\mapsto \left(\begin{pmatrix}aI&bI\\cI&dI\end{pmatrix}, f\right)$$
Where $\sigma$ above is defined in coordinates $V\oplus V^*\oplus V\oplus V^*$, $I$ being the identity matrix of $V\oplus V^*$, and $f$ is defined by $f(u,u',v,v')=ac<u,u'>+bd<v,v'>+bc(<u,v'>+<v,u'>)$.
\end{proposition}
\begin{proof}
Indeed:
\begin{align*}
&B(\sigma(u_1, u'_1, v_1, v'_1), \sigma(u_2, u'_2, v_2, v'_2)) + B((u_1, u'_1, v_1, v'_1), (u_2, u'_2, v_2, v'_2)) =\\&= B((au_1+bv_1, au'_1+bv'_1, cu_1+dv_1, cu'_1+dv'_1),(au_2+bv_2, au'_2+bv'_2, cu_2+dv_2, cu'_2+dv'_2)) +\\&+ B((u_1, u'_1, v_1, v'_1), (u_2, u'_2, v_2, v'_2)) =\\&= <au_1+bv_1,cu'_2+dv'_2>+<cu_2+dv_2,au'_1+bv'_1>+<u_1,v'_2>+<v_2,u'_1>=\\&=ac(<u_1,u'_2>+<u_2,u'_1>)+bd(<v_1,v'_2>+<v_2,v'_1>) + bc(<v_1,u'_2>+<u_2,v'_1>)+\\&+(ad+1)(<u_1,v'_2>+<v_2,u'_1>) =\\&= ac(<u_1,u'_2>+<u_2,u'_1>)+bd(<v_1,v'_2>+<v_2,v'_1>) +\\&+ bc(<v_1,u'_2>+<u_2,v'_1>+<u_1,v'_2>+<v_2,u'_1>) =\\&= f((u_1, u'_1, v_1, v'_1)+(u_2, u'_2, v_2, v'_2)) + f((u_1, u'_1, v_1, v'_1)) + f((u_2, u'_2, v_2, v'_2))
\end{align*}
To check that this is a morphism of groups, one needs to check that given $g_1,g_2\in SL_2(\FF)$ which map into $(\sigma_1,f_1),(\sigma_2,f_2)$, the quadratic form associated to $\sigma_1\sigma_2$ is $f_1\sigma_2+f_2$.
Indeed
\begin{align*}
&(f_1\sigma_2+f_2)(u,u',v,v')=\\&=
a_1c_1<a_2u+b_2v,a_2u'+b_2v'>+b_1d_1<c_2u+d_2v,c_2u'+d_2v'> +\\&+ b_1c_1(<a_2u+b_2v,c_2u'+d_2v'>+<c_2u+d_2v,a_2u'+b_2v'>) +\\&+
a_2c_2<u,u'>+b_2d_2<v,v'>+b_2c_2(<u,v'>+<v,u'>) =\\&=
(a_1c_1a_2^2+b_1d_1c_2^2+a_2c_2)<u,u'>+(a_1c_1b_2^2+b_1d_1d_2^2+b_2d_2)<v,v'>+\\&+(a_1c_1a_2b_2+b_1d_1c_2d_2+b_1c_1a_2d_2+b_1c_1b_2c_2+b_2c_2)(<u,v'>+<v,u'>) =\\&=
(a_1c_1a_2^2+b_1d_1c_2^2+a_2c_2a_1d_1+a_2c_2b_1c_1)<u,u'>+\\&+(a_1c_1b_2^2+b_1d_1d_2^2+a_1d_1b_2d_2+b_1c_1b_2d_2)<v,v'>+\\&+(a_1c_1a_2b_2+b_1d_1c_2d_2+b_1c_1a_2d_2+a_1d_1b_2c_2)(<u,v'>+<v,u'>) =\\&=
(a_1a_2+b_1c_2)(c_1a_2+d_1c_2)<u,u'>+(a_1b_2+b_1d_2)(c_1b_2+d_1d_2)<v,v'>+\\&+(a_1b_2+b_1d_2)(c_1a_2+d_1c_2)(<u,v'>+<v,u'>)
\end{align*}
which is indeed the quadratic form associated to $\sigma_1\sigma_2$.
\end{proof}
\begin{theorem}[see \cite{Weil} and \cite{Bla}]
Fix an additive character $\psi$ of $\FF$.
There is a projective representation $\rho$ of $PSp$ on $\Sc(X)$, such that for any $a,b\in PSp$, we have $\rho(ab)^{-1}\rho(a)\rho(b)=\pm 1$. We also have the explicit formulas, with $x=(v,v')\in X$:
$$\left(\left(\rho j\begin{pmatrix}t&0\\0&t^{-1}\end{pmatrix}\right)\Phi\right)(x)=|t|^{\dim V}\Phi(tx),$$
$$\left(\left(\rho j\begin{pmatrix}1&u\\0&1\end{pmatrix}\right)\Phi\right)(x)=\psi(u<v,v'>)\Phi(x),$$
$$\left(\left(\rho j\begin{pmatrix}0&1\\1&0\end{pmatrix}\right)\Phi\right)(x)=\Fou(\Phi)(x).$$
In the last equation $\Fou$ denotes the Fourier transform on $X$ with respect to the symmetric bilinear non-degenerate form $((u,u'),(v,v')):=<u,v'>+<v,u'>$.
\end{theorem}
\begin{proof}[Proof of Theorem \ref{RS thm}]
The projective action of $PSp$ on $\Sc(V\oplus V^*)$ gives a projective representation of $PSp$ on $\Sc^*(V\oplus V^*)$.
Looking in the formulas for this action, we see that if $\xi\in \Sc^*(V\oplus V^*)$ is as in the formulation of the theorem, we have
$$\rho j\begin{pmatrix}1&u\\0&1\end{pmatrix}(\xi)=\pm \xi$$
and also
$$\rho j \begin{pmatrix}1&0\\u&1\end{pmatrix}(\xi)=\rho j\left(\begin{pmatrix}0&1\\1&0\end{pmatrix}\begin{pmatrix}1&u\\0&1\end{pmatrix}\begin{pmatrix}0&1\\1&0\end{pmatrix}\right)(\xi)=\pm\xi.$$ Since elements of the form $\begin{pmatrix}1&u\\0&1\end{pmatrix}$ and $\begin{pmatrix}1&0\\u&1\end{pmatrix}$ generate $SL_2(\FF)$, we get that $\xi$ is $SL_2(\FF)$ $\pm$-invariant. I particular $\rho j \begin{pmatrix}t&0\\0&t^{-1}\end{pmatrix}(\xi)=\pm\xi$. So we get that $$|t|^{-\dim V}\cdot |(t\xi)(\Phi)| = |\xi(\Phi)|$$
as desired.
\end{proof}
\bibliographystyle{acm}
\bibliography{citations}

\begin{thebibliography}{1}

\bibitem{AGS}
{\sc Aizenbud, A., Avni, N., and Gourevitch, D.}
\newblock Spherical pairs over close local fields.
\newblock {\em Commentarii Mathematici Helvetici 87}, 4 (2012), 929--962.

\bibitem{AG}
{\sc Aizenbud, A., and Gourevitch, D.}
\newblock A proof of the multiplicity one conjecture for $gl_n$ in $gl_{n+1}$.
\newblock https://arxiv.org/abs/0707.2363v2.

\bibitem{AG-R}
{\sc Aizenbud, A., and Gourevitch, D.}
\newblock Multiplicity one theorem for
  $(\mathrm{GL}_{n+1}(\mathbb{R}),\mathrm{GL}_n(\mathbb{R}))$.
\newblock {\em Selecta Mathematica. New Series 15\/} (2009).

\bibitem{AGRS}
{\sc {Aizenbud}, A., {Gourevitch}, D., {Rallis}, S., and {Schiffmann}, G.}
\newblock {Multiplicity one theorems}.
\newblock {\em {Ann. Math. (2)} 172}, 2 (2010), 1407--1434.

\bibitem{Bla}
{\sc Blasco, L.}
\newblock Paires duales réductives en caractéristique 2.
\newblock {\em Mémoires de la Société Mathématique de France 52\/} (1993),
  1--73.

\bibitem{M}
{\sc Mezer, D.}
\newblock Multiplicity one theorem for $(\mathrm {GL}_{n+1},\mathrm {GL}_n)$
  over a local field of positive characteristic.
\newblock {\em Mathematische Zeitschrift 297\/} (04 2021).

\bibitem{RS}
{\sc Rallis, S., and Schiffmann, G.}
\newblock Multiplicity one conjectures.
\newblock https://arxiv.org/abs/0705.2168.

\bibitem{SZ}
{\sc Sun, B., and Zhu, C.-B.}
\newblock Multiplicity one theorems: the {A}rchimedean case.
\newblock {\em Annals of Mathematics 175\/} (2009).

\bibitem{Weil}
{\sc Weil, A.}
\newblock {Sur certains groupes d'opérateurs unitaires}.
\newblock {\em Acta Mathematica 111}, none (1964), 143 -- 211.

\end{thebibliography}
\end{document}